\documentclass{birkjour}

\newcommand{\qhypergeom}[5]{\mbox{$
_#1 \phi_#2\left. \left(\!\!\!\!
\begin{array}{c}
\multicolumn{1}{c}{\begin{array}{c} #3
\end{array}}\\[1mm]
\multicolumn{1}{c}{\begin{array}{c} #4
           \end{array}}\end{array}
\right| \displaystyle{#5}\right) $} }
 \newtheorem{theorem}{Theorem}[section]

 \newtheorem{proposition}[theorem]{Proposition}
 \theoremstyle{definition}
 
 \theoremstyle{remark}

 \numberwithin{equation}{section}

\begin{document}

%
%
%
%
%
%
%
%
%

\title[Characterization of certain sequences  of $q$-polynomials]
 {Characterization of certain sequences \\ of $q$-polynomials}

\author[P. Njionou Sadiang]{P. Njionou Sadjang}

\address{%
University of Douala,\\
Faculty of Industrial Engineering\\
Douala\\
Cameroon}

\email{pnjionou@yahoo.fr}

\thanks{This work was completed with the support of the University of Kassel
}



\date{\today}

\begin{abstract}
A characterization is given of those sequences of quasi-orthogonal polynomials which form also $q$-Appell sets. 
\end{abstract}

\maketitle

\section{Introduction}

\noindent A sequence of polynomials $\{Q_n\}$, $n=0,\; 1,\; 2,\; \ldots $, $\deg Q_n(x)=n$ is said to be quasi-orthogonal if there is an interval $(a,b)$ and a non-decreasing function $\alpha(x)$ such that 
\[ \int_a^b x^mQ_n(x)d\alpha(x)\left\{\begin{array}{ll}
=0 & \textrm{for}\quad 0\leq m\leq n-2\\
\neq 0 &\textrm{for}\quad 0\leq m=n-1\\
\neq 0 &\textrm{for}\quad 0=m=n. 
\end{array}\right.\]
We say that two polynomial sets are related if one set is quasi-orthogonal with respect to the interval and the distribution of the orthogonality of the other set.
Riesz \cite{riesz} and Chihara \cite{chihara} have shown that a necessary and sufficient condition for the quasi-orthogonality of the $\{Q_n(x)\}$ is that there exist nonzero constants, $\{a_n\}_{n=0}^{\infty}$ and $\{b_n\}_{n=1}^{\infty}$, such that 
\begin{equation}
   \begin{array}{lll}
      Q_n(x)&=&a_n P_n(x)+b_n P_{n-1}(x),\\
      Q_0(x)&=& a_0P_0(x)
   \end{array}\quad n\geq 1,
\end{equation}
where the $\{P_n(x)\}_{n=0}^{\infty}$ are the related orthogonal polynomials. 

\noindent  A polynomial set $\{P_n(x)\}$ forms a $q$-Appell set if  the condition \cite{sharma,al-salam} given by
\begin{equation}
  D_qP_n(x)=[n]_q P_{n-1}(x),\quad n=0,\, 1,\; 2,\; 3,\;\ldots
\end{equation}
is fulfilled. 
Here, $D_q$ is the Hahn derivative defined by 
\[D_qf(x)=\dfrac{f(x)-f(qx)}{(1-q)x},\quad x\neq 0\]
and $[n]_q$ stands for the $q$-number defined by $[n]_q=\dfrac{1-q^n}{1-q}$.
Such sets were first introduced by Sharma and Chak \cite{sharma} to generalize the Appell set \cite{appell}. Several papers related to $q$-Appell sets have been recently published, see for instance \cite{ernst2015-1} and the references therein. 

In 1967, Al-Salam has given in \cite{al-salam} a characterization of those sequences of orthogonal polynomials $\{P_n(x)\}$ which are also $q$-Appell sets. More precisely, he gave a characterization of those sequences of orthogonal polynomials for which $D_qP_n(x)=[n]_qP_{n-1}(x)$ for $n=0,\, 1,\; 2,\; 3,\;\ldots$. 

The purpose of this paper is to study  those classes of  polynomial sets $\{P_n(x)\}$ that are at the same time quasi-orthogonal sets and $q$-Appell sets.

\section{Preliminaries results and definitions}

\noindent Let us introduce the so-called $q$-Pochhammer symbol
\[(x;q)_n=\left\{
\begin{array}{ll}
(1-x)(1-xq)\cdots(1-xq^{n-1}) & n=1,2,\dots\\
1& n=0
\end{array}
\right.\] 

\noindent and
 the basic hypergeometric or $q$-hypergeometric function $_r\phi_s$ is defined by the series
 \[
 \qhypergeom{r}{s}{a_1,\ldots,a_r}{b_1,\ldots,b_s}{q;z}:=\sum_{k=0}^\infty\frac{(a_1,\ldots,a_r;q)_k}{(b_1,\cdots,b_s;q)_k}\left((-1)^k
 q^{\binom{k}{2}}\right)^{1+s-r}\frac{z^k}{(q;q)_k},
 \]
  where
 \[(a_1,\ldots,a_r)_k:=(a_1;q)_k\ldots(a_r;q)_k. \]
\noindent The Al Salam-Carlitz I polynomials \cite[P.\ 534]{KLS} have the $q$-hypergeometric representation 
\begin{eqnarray*}
U_n^{(a)}(x;q)&=&(-a)^nq^{n\choose2}\qhypergeom{2}{1}{q^{-n},x^{-1}}{0}{q;\frac{qx}{a}}
\end{eqnarray*}
and fulfill the three-term recurrence relation
\begin{equation*}
xU_n^{(a)}(x;q)=U_{n+1}^{(a)}(x;q)+(a+1)q^nU_n^{(a)}(x;q)-aq^{n-1}(1-q^n)U_{n-1}^{(a)}(x;q) 
\end{equation*}
and the $q$-derivative rule 
\begin{equation*}
D_q U_n^{(a)}(x;q)=[n]_qU_{n-1}^{(a)}(x;q).
\end{equation*}
It is therefore clear that the Al-Salam Carlitz I polynomials form a $q$-Appell set. 

\begin{proposition}(See \cite[Theorem 1]{dickinson})\label{dickinson}
For $\{Q_n(x)\}$ to be a set of polynomials quasi-orthogonal with respect to an interval $(a,b)$ and a distribution $d\alpha(x)$, it is necessary and sufficient that there exist a set of nonzero constants $\{T_k\}_{k=0}^{\infty}$ and a set of polynomials $\{P_n(x)\}$ orthogonal with respect to $(a,b)$ and $d\alpha(x)$ such that 
\begin{equation}\label{connect0}
  P_n(x)=\sum_{k=0}^{n}T_kQ_k(x),\quad n\geq 0.
\end{equation}
\end{proposition}

\begin{proposition}(See \cite[Theorem 2]{dickinson})\label{dick2}
A necessary and sufficient condition that the set $\{Q_n(x)\}_{n=0}^{\infty}$ where each $Q_n(x)$ is a polynomial of degree precisely $n$, be quasi-orthogonal is that it satisfies 
\begin{equation*}
Q_{n+1}(x)=(x+b_n)Q_n(x)-c_nQ_{n-1}(x)+d_n\sum_{k=0}^{n-2}T_kQ_k(x),
\end{equation*}
for all $n$, with $d_0=d_1=0$. 
\end{proposition}


\begin{proposition}(See \cite[Theorem 4.1]{al-salam})\label{pro1}
If $\{Q_n(x)\}_{n=0}^{\infty}$ is a $q$-Appell set which are also orthogonal, then there exists a non zero constant $b$ such that 
\[Q_n(x)=b^nU_{n}^{(a/b)}\left(\dfrac{x}{b};q\right)\]
for all $n\geq 0$.  
\end{proposition}

\section{A characterization result}

\noindent In this section, we characterize quasi-orthogonal polynomial sets that are also $q$-Appell sets. 

\begin{theorem}\label{theo-car1}
If $\{Q_n(x)\}_{n=0}^{\infty}$ is a $q$-Appell set of quasi-orthogonal polynomials,
 then there exist three reel numbers $b$, $c$ and $\lambda$, such that 
\begin{equation}\label{rec-car1}
 Q_{n+1}(x)=(x+bq^n)Q_n(x)-cq^{n-1}[n]_qQ_{n-1}(x)+d_n\sum_{k=0}^{n-2}\dfrac{\lambda^k}{[k]_q!}Q_k(x). 
\end{equation}
\end{theorem}

\begin{proof}
Assume that $\{Q_n(x)\}_{n=0}^{\infty}$ is a $q$-Appell set which is quasi-orthogonal and $\{P_n(x)\}_{n=0}^{\infty}$ the related orthogonal family. From Proposition \ref{dick2}, there exist three sequences $\{a_n\}_{n=0}^{\infty}$, $\{b_n\}_{n=0}^{\infty}$ and $\{d_n\}_{n=0}^{\infty}$ with $d_0=d_1=0$ such that 
\begin{equation}\label{rr2}
Q_{n+1}(x)=(x+b_n)Q_n(x)-c_nQ_{n-1}(x)+d_n\sum_{k=0}^{n-2}T_kQ_k(x).
\end{equation}
If we $q$-differentiate \eqref{rr2} and use the fact that $\{Q_n(x)\}_{n=0}^{\infty}$ is a $q$-Appell, we get after some simplifications 
\begin{equation}\label{rrr1}
Q_n(x)=\left(x+\dfrac{b_n}{q}\right)Q_{n-1}(x)-\dfrac{c_n}{q}\dfrac{[n-1]_q}{[n]_q}Q_{n-2}(x)+\dfrac{d_n}{q[n]_q}\sum_{k=0}^{n-3}[k+1]_qT_{k+1}Q_k(x). 
\end{equation}
Next, if we replace $n$ by $n-1$ in \eqref{rr2}, we obtain 
\begin{equation}\label{rrr2}
Q_{n}(x)=(x+b_{n-1})Q_{n-1}(x)-c_{n-1}Q_{n-2}(x)+d_{n-1}\sum_{k=0}^{n-3}T_kQ_k(x).
\end{equation}
If we compare \eqref{rrr1} and \eqref{rrr2}, we see that we should have 
\begin{equation}\label{coef1}
b_{n}=qb_{n-1},\quad c_{n}=q\dfrac{[n]_q}{[n-1]_q}c_{n-1},
\end{equation}
and 
\begin{equation}\label{coef2}
 d_n [k+1]_qT_{k+1}=q[n]_qd_{n-1}T_k,\quad\quad k=0,\; 1,\;\cdots n-3.
\end{equation}
Equations \eqref{coef1} yield
\[b_n=q^n b_0,\quad\textrm{and}\quad c_n=q^{n-1}[n]_qc_1.\]
Next, \eqref{coef2} gives for $k=0$ and $k=n-3$ the relations
\begin{equation}\label{coef3}
d_n=\dfrac{q[n]_q}{T_1}d_{n-1}\quad \textrm{and}\quad d_n=\dfrac{q[n]_qT_{n-1}}{[n-2]_qT_{n-2}}d_{n-1}.
\end{equation}
If, for a given $k\geq 2$, $d_k=0$, it follows from \eqref{coef3} that $d_k=0$ for all $k$. In this case \eqref{rr2} becomes a three-term recurrence relation 
\begin{equation}\label{ttr2}
Q_{n+1}(x)=(x+b_n)Q_n(x)-c_nQ_{n-1}(x).
\end{equation} 
In this case, from Proposition \ref{pro1}, it is seen that $\{Q_n(x)\}_{n=0}^{\infty}$ is essentially the sequence of Al-Salam Carlitz polynomials. Thus, in this case, $\{Q_n(x)\}_{n=0}^{\infty}$ is not a sequence of quasi-orthogonal polynomials. Thus, we must have $d_k\neq 0$ for $k\geq 2$. 

Again, using \eqref{coef3}, we have for all $n\geq 0$ the identity $\dfrac{T_{n-1}}{[n]_qT_n}=\dfrac{1}{T_1}$. This last relation gives 
$T_n=\dfrac{T_1^n}{[n]_q!}$. Seting $b_0=b$, $c_1=c$ and $T_1=\lambda$, this ends the proof of the theorem. 
\end{proof}

\begin{theorem}\label{theo-car2}
Let $\{Q_n(x)\}_{n=0}^{\infty}$ be a monic polynomial set. The following assertions are equivalent:
\begin{enumerate}
   \item $\{Q_n(x)\}_{n=0}^{\infty}$ is quasi-orthogonal and is a $q$-Appell set, $n\geq 1$.
   \item There exists three constants $\alpha$, $\beta$ and $\gamma$ ($\beta,\gamma\neq 0$) such that 
   \[ Q_n(x)=\beta^nU_n^{(\alpha/\beta)}\left(\dfrac{x}{\beta};q\right)- \dfrac{\beta^n [n]_q}{\lambda}U_{n-1}^{(\alpha/\beta)}\left(\dfrac{x}{\beta};q\right)\;\quad (n\geq 1),\]
   where $U_n^{(a)}(x;q)$ are the Al-Salam Carlitz I polynomials.  
\end{enumerate}
\end{theorem}

\begin{proof}
Suppose first that $\{Q_n(x)\}_{n=0}^{\infty}$ is quasi-orthogonal and is a $q$-Appell set, $n\geq 1$. Then, by Theorem \ref{theo-car1}, the $Q_n$'s satisfy a recurrence relation of the form \eqref{rec-car1}. Let us define the polynomial set $\{P_n(x)\}_{n=0}^{\infty}$ by
\begin{equation}\label{rel-pol}
P_n(x)=\dfrac{[n]_q!}{\lambda^n}\sum_{k=0}^{n}\dfrac{\lambda^k}{[k]_q!}Q_k(x). 
\end{equation} 
It is not difficult to see that 
\[ D_q P_n(x)=\dfrac{[n]_q!}{\lambda^n}\sum_{k=1}^{n}\dfrac{\lambda^k}{[k]_q!}[k]_qQ_{k-1}(x)= [n]_q\dfrac{[n-1]_q!}{\lambda^{n-1}}\sum_{k=0}^{n-1}\dfrac{\lambda^k}{[k]_q!}Q_k(x)=[n]_q P_{n-1}(x).\]
Hence, $\{P_n(x)\}_{n=0}^{\infty}$ is a $q$-Appell set. Moreover, $\{P_n(x)\}_{n=0}^{\infty}$ is the orthogonal set (see Proposition \ref{dickinson}) related to $\{Q_n(x)\}_{n=0}^{\infty}$. By Proposition \ref{pro1}, there exist $\alpha$ and $\beta$ such that 
\[P_n(x)=\beta^nU_n^{(\alpha/\beta)}\left(\dfrac{x}{\beta};q\right).\]
Next, from \eqref{rel-pol}, it follows easily that 
\[Q_n(x)=\dfrac{[n]_q!}{\lambda^n}\left(P_n(x)-P_{n-1}(x)\right).\]
The first implication of the theorem follows. \\
Conversely, assume that  there exist three constants $\alpha$, $\beta$ and $\gamma$ ($\beta,\gamma\neq 0$) such that 
 \[ Q_n(x)=\beta^nU_n^{(\alpha/\beta)}\left(\dfrac{x}{\beta};q\right)- \dfrac{\beta^n [n]_q}{\lambda}U_{n-1}^{(\alpha/\beta)}\left(\dfrac{x}{\beta};q\right)\;\quad (n\geq 1).\]
   It is easy to see that $\{Q_n(x)\}_{n=0}^{\infty}$ is a quasi-orthogonal set. It remains to prove that $\{Q_n(x)\}_{n=0}^{\infty}$ is a $q$-Appell set. Using the fact that $D_q [f(a x)]=a [D_qf](ax)$, we have 
   \[D_q U_n^{(\alpha/\beta)}\left(\dfrac{x}{\beta};q\right)=\dfrac{1}{\beta} U_{n-1}^{(\alpha/\beta)}\left(\dfrac{x}{\beta};q\right).\] It follows that $D_q Q_n(x)=[n]_q Q_{n-1}(x)$. This ends the proof of the theorem.
\end{proof}




\end{document}